\newcommand{\Coker}{\mathbf{c}}
 \newcommand{\CC}{\mathcal{C}}

 \newcommand{\uparameter}{u}
 \newcommand{\rparameter}{r}
 \newcommand{\Cr}{\mathcal{C}_\uparameter}

\newcommand{\doublematrix}[4]{
\begin{mat}{ccc|ccc}
&&&&&\\
&#1&&&#2&\\
&&&&&\\ \hline
&&&&&\\
&#3&&&#4&\\
&&&&&
\end{mat}
}

\newcommand{\onedoublematrix}[3]{
\begin{mat}{ccc|ccc}
&&&&&\\
&&&&#2&\\
&#1&&&&\\ \cline{4-6}
&&&&&\\
&&&&#3&\\
&&&&&
\end{mat}
}

\newcommand{\halfrandommatrix}[1]{
\begin{mat}{ccc|ccc}
&&&*&\hdots&*\\
&&&\vdots&\ddots&\vdots\\
&#1&&*&\hdots&*\\ \cline{4-6}
&&&*&\hdots&*\\
&&&\vdots&\ddots&\vdots\\
&&&*&\hdots&*
\end{mat}
}

\newcommand{\randommatrix}{
\begin{mat}{ccc}
*&\hdots&*\\
\vdots&\ddots&\vdots \\
*&\hdots&*\\ \hline
*&\hdots&*\\
\vdots&\ddots&\vdots \\
*&\hdots&*
\end{mat}
}

\newcommand{\triplematrix}[9]{
\begin{mat}{ccc|ccc|ccc}
&&&&&&&&\\
&#1&&&#2&&&#3&\\
&&&&&&&&\\ \hline
&&&&&&&&\\
&#4&&&#5&&&#6&\\
&&&&&&&& \\ \hline
&&&&&&&&\\
&#7&&&#8&&&#9&\\
&&&&&&&& \\ 
\end{mat}
}

\title{Markov chains arising in the study of random matrices over pro-finite rings}

\documentclass{article}

\usepackage{ifxetex,ifluatex}
\if\ifxetex T\else\ifluatex T\else F\fi\fi T%
  \usepackage{fontspec}
\else
  \usepackage[T1]{fontenc}

  \usepackage{blkarray, bigstrut}
  \usepackage[utf8]{inputenc}
  
  \usepackage{amsthm}
  \usepackage{thmtools}
  \usepackage{mathtools}

  \usepackage{lmodern}
  \usepackage{amssymb}
  \usepackage{amsfonts}
  \usepackage{tikz-cd}
  \usepackage{amsthm}
  \usepackage{xfrac}
  \usepackage{mathtools}
  \usepackage{multirow}
  \usepackage{mathrsfs}
  \usepackage{comment}

  \newtheorem*{corollary}{Corollary}

  \newcommand{\probP}{\text{I\kern-0.15em P}}
  \newcommand{\probE}{\text{I\kern-0.15em E}}
  \newcommand{\coker}{\text{coker}}
  \newcommand{\defeq}{\overset{\mathrm{def}}{=\joinrel=}}

  \numberwithin{equation}{section}

  \theoremstyle{remark}
  \newtheorem*{remark}{Remark}
  \theoremstyle{remark}

  \newtheorem*{observation}{Observation}
  \newtheorem*{definition}{Definition}

  \newcommand{\Z}{\mathbb{Z}}

  \excludecomment{mysection}
  \excludecomment{mymysection}
  \excludecomment{maybeinclude}

    \newenvironment{mat}
    {
    \left[
\begin{array}}
{ \end{array}
\right]
}

\newcommand{\probM}{\mathcal{M}}

\newtheorem{Remark}{Remark}[section]

\usepackage{mdframed}

\theoremstyle{theorem}
\newtheorem{claim}{Claim}

\newtheorem{theorem}{Theorem}

\newtheorem*{theorem*}{Theorem}
\newtheorem{lemma}{Lemma}[theorem]
\newtheorem{sublemma}{Lemma}[lemma]
\newenvironment{ftheo*}
  {\begin{mdframed}\begin{theorem*}}
  {\end{theorem*}\end{mdframed}}

\newtheorem{Corollary}{Corollary}[theorem]

\newtheorem*{notation}{Notation}
\newtheorem*{goal}{Goal}

  \usepackage{enumitem}

   \DeclareSymbolFont{bbold}{U}{bbold}{m}{n}
   \DeclareSymbolFontAlphabet{\mathbbold}{bbold}

   \theoremstyle{remark}

\fi

 \theoremstyle{definition}

 \newcommand{\comm}[1]{} 

 \newcommand{\prob}{\probP}

 \usepackage{cite}
 \newcommand{\commm}[1]{}
 
 \newcommand{\chapter}{section}

  \usepackage{
  hyperref,
  cleveref,
  }
\author{Nikita Lvov\footnote{nikita.lvov@mail.mcgill.ca}
} 
\begin{document} 
\maketitle

\begin{abstract}
Recent work of the author investigates certain random processes, valued in abelian $p$-groups, that naturally arise in the study of Haar random matrices over $\Z_p$. It was found, somewhat surprisingly, that these processes are reversible Markov chains. In this short note, we give a simple alternative derivation of this fact. The new derivation also proves that this phenomenon is not specific to $\Z_p$, but generalizes to Haar random matrices over any profinite local ring.
\end{abstract}

\section{Introduction}

\begin{mysection}
Let $M$ be a matrix over a finite field $k$. Let $\probM$ be the random matrix obtained from $M$ by adding a column vector whose entries are independent, uniformly distributed random variables in $k$:
\[
\probM \defeq 
\begin{mat}{ccc|c}
&&&*\\
&M&&\vdots\\
&&&*
\end{mat}
\]
\end{mysection}

\begin{mysection}
Suppose that $M$ is a square matrix over a finite field. How does the corank of $M$ change, when we add a random column vector and a random row vector to $M$? The corank either increases by $1$, decreases by $1$ or 
\end{mysection}

Perhaps, the first and simplest instance of a Markov chain appearing in the study of random matrices, is the evolution of the corank of the principal minors of a large random matrix over a finite field.

\paragraph{Random matrices over finite fields} Let $k$ be a finite field. Consider an infinite matrix $\probM$, whose entries are indexed by $\mathbb{N} \times \mathbb{N}$ and are independent \textit{uniformly distributed} random variables in $k$. Let $\probM_{n,n}$ be the top-left $n \times n$ corner of $\probM$. Then
\begin{equation}
\label{eqn:process}
corank(\probM_{n,n})
\end{equation} 
is a random process valued in $\mathbb{N}$. Using standard arguments in linear algebra, one can deduce the following interesting features of this process:
\begin{itemize}
\item (\ref{eqn:process}) is a Markov chain, i.e. a random walk on $\mathbb{N}$ (with certain explicit transition propabilities, that depend on the cardinality of $k$.) 
\item This Markov chain is recurrent, and hence converges to its unique stationary distribution.
\end{itemize}

Finally, we remark that this Markov chain is also \textit{reversible}, for trivial reasons, since every recurrent Markov chain on $\mathbb{N}$, for which $p(i \rightarrow j) \neq 0$ if $|i-j|=1$ and $p(i \rightarrow j) = 0$ if $|i-j|\geq2$, is reversible by the Kolmogorov cycle condition.

\paragraph{Random matrices over $\Z_p$}
A non-trivial generalization of the above facts to matrices over $\Z_p$ was studied in \cite{arxiv}.

\begin{theorem*} \cite[Theorem 1.1]{arxiv}
\begin{equation}
\label{eqn:cloperator}
coker \Big( Haar_{n,n}(\Z_p) \Big) 
\end{equation} is a \textit{reversible} Markov chain on abelian $p$-groups.
\end{theorem*}

\begin{remark}
The process (\ref{eqn:cloperator}) is a direct generalization of the process (\ref{eqn:process}). Indeed, the cokernel is a direct generalization of the corank, as for a matrix $M$ over a finite field $k$, \[ coker(M) \cong k^{corank(M)}. \] 
\end{remark}

In fact, in \cite[Theorem 1.1]{arxiv}, we describe the generator of this Markov chain explicitly. In its most elementary form, the Markov chain (\ref{eqn:cloperator}) is generated by the following operation:
\begin{itemize}
\item[(A)] Given an abelian $p$-group $G$, mod out by a uniformly random element of $G$.
\item[(B)] Take the product of $\Z_p$.\footnote{Steps (A) and (B) together are equivalent to the following single step - pick a uniformly random extension of $G$ by $\Z_p$. The equivalence is a consequence of \cite[Lemma 1.3]{arxiv}.}
\item[(C)] Finally, mod out by a Haar random element of this product.
\end{itemize}
Futhermore, in \cite[Lemma 2.3]{arxiv} we construct an explicit weighted graph $\Gamma$, such that the Markov chain is realized as a symmetric random walk on $\Gamma$. Finally, the last chapter of \cite[Section 3]{arxiv}, together with \cite{VanPeskiLvov}, gives a complete explicit solution to the spectral problem for this Markov chain.

\paragraph{Main results of this paper: random matrices over pro-finite local rings and a unified approach}

The theorem \cite[Theorem 1.1]{arxiv} is unsatisfactory in two respects:
\begin{itemize}
\item It is not a priori clear why the Markov property and the reversibility property hold. 
\begin{itemize}
\item There are related Markov chains that arise in the study of symmetric and anti-symmetric matrices over $\Z_p$. These Markov chains are also turn out to reversible. In each case, this is shown by a computation. However, the  seeming ubiquity of this phenomenon begs for a unified approach.
\end{itemize}
\item We can ask what happens for higher-dimensional rings. The approach of \cite{arxiv} does not directly generalize to this case.
\begin{itemize}
\item Aside from intrinsic interest, such a gneralization would be desirable, as $2$-dimensional rings (e.g. $\Z_p[x]$) arise even in the study of $p$-adic random matrices. For example, to understand the characterstic polynomial of a $p$-adic matrix $\probM$, one is led to consider $det(\probM-Ix)$, which is most naturally understood through $coker(\probM-Ix)$, the cokernel of a random matrix over $\Z_p[x]$.\footnote{Of course, this matrix is not Haar random, but many classes of random matrices can be proven to behave asymptotically in the same way as Haar random matrices. This is the universality phenomenon, that is discussed in particular in \cite[Section 3]{WoodICM}.}
\end{itemize}
\end{itemize}

In this paper we prove the following theorem.

\begin{theorem}
\label{thm:maintheorem}
Let $R$ be a finite or profinite local ring. Then $coker(Haar_{n,n+u}(R))$ is a reversible Markov chain, for any $u \in \Z$.
\end{theorem}

\begin{mysection}
\item $coker(Haar_{l\lfloor\frac{n+1}{2}\rfloor+k_2 ,l\lfloor\frac{n}{2}\rfloor+k_1}(R))$ is a reversible Markov chain, for any $l \in \Z$ and 
\item More generally, define the following sequence of random matrices $\probM_n$
\end{mysection}

\begin{remark}
In fact, we prove more. The random variable $coker(Haar_{n,m}(R))$ is a "two-dimensional" Markov chain, governed by a certain family of transition operators. We refer to \S \ref{sec:randomoperators} for more details.
\end{remark}

\begin{remark}
By itself, the \textit{Markov property} is not surprising, and can be deduced from the following two facts:
\begin{itemize}
\item Suppose that a module is defined as a set of generators satisfying linear relations. Adding another Haar random relation has the effect of modding out by a Haar random element of this module.
\item The module $\coker(M)$ is uniquely determined by $\coker(M^T)$ and the dimensions of $M$.
\end{itemize}
Indeed, the first fact tells us what happens to the cokernel when we add a Haar random column vector to a matrix. The second fact shows that this infomation is sufficient to determine what happens when we add a Haar random row vector.
\end{remark}

However, the approach in this paper can be generalized to prove the Markov property, as well as reversibility, for \textit{symmetric} matrices over $R$. In this situation, the origin of the Markov property is a priori much less evident.

\begin{remark}
In this paper, contrary to \cite{arxiv}, the generator of the Markov chain is not described explicitly. Rather, the Markov property is deduced from invariance properties of the Haar measure, as we briefly elaborate below.
\end{remark}

\subsection{Approach}
\label{sec:approach}

Here, we say a few words about our approach. We will define an equivalence relation on matrices over $R$, not necessarily square, of arbitrary dimension:
\[M_1 \sim M_2 \]
if and only if there exist $n$ and $m$ such that:

\begin{equation}
\label{eqn:eqrelation}
\begin{mat}{c|c}
M_1& 0 \\ \hline
0  & \mathbf{I}_n
\end{mat}=
g^T
\begin{mat}{c|c}
M_2& 0 \\ \hline
0  & \mathbf{I}_m
\end{mat}g' \hspace{0.7in} \text{ for some } g,g' \in GL(R)
\end{equation}

This equivalence relation has the property that two matrices are equivalent if and only if they have the same cokernel \textit{and} the difference between the number of columns and the number of rows is the same for both matrices.

However, the equivalence relation turns out to be more convenient to work with, without any reference to the cokernel. Indeed, it follows almost immediately from the invariance properties of the Haar measure that
\[
Haar_{n,n+u}(R)
\]
induces a reversible Markov chain on equivalence classes of matrices, under the equivalence relation (\ref{eqn:eqrelation}). \autoref{thm:maintheorem} can then be deduced as a corollary.

\begin{mysection}
\begin{remark}
Thus, two matrices can have the same cokernel, yet lie in different classes. Moreover, it is incorrect to say that equivalence classes are in bijection with 
\[
R-mod \times \Z
\]
Indeed, the minimal presentation of a given module can have many more relations than generators.
\end{remark}
\end{mysection}

\begin{itemize}
\item[] There are two other reasons why it is preferable to work with equivalence classes of matrices rather than the modules that parametrize them.

\begin{itemize}
\item This point of view will add flexibility, as we can mildly modify the equivalence relation. Indeed, although we do not elaborate on this in the present paper, all the results go through if we consider a weaker equivalence relation than (\ref{eqn:eqrelation}), namely if we replace the group $GL(R)$ by $SL(R)$.

\item Most importantly, the point of view also has the potential of readily generalizing to the study of matrices with symmetry, such as symmetric, anti-symmetric, or Hermitian matrices. In these cases, the equivalence classes are not parametrized by modules, but by modules with extra information. Rather than pin down this extra information, it seems more convenient to work with the equivalence classes directly.
\end{itemize}
\end{itemize}


We will write $\Coker(M)$ to denote the equivalence class of $M$.

\subsection{Outline}

In \S \ref{sec:eqrelation}, we make some remarks on the equivalence relation (\ref{eqn:eqrelation}). In \S \ref{sec:cokernel}, we show that the equivalence class, to which a matrix $M$ belongs, is determined by $coker(M)$ and by the difference between the number of rows and the number of columns of $M$. In \S \ref{sec:markovchain}, we show that the process induced by $Haar_{n.n+u}$ on equivalence classes is a Markov chain. In \S \ref{sec:recurrence}, we show that this Markov chain is recurrent. This is done by reducing to the finite field case. We also draw a connection to the recent work of Sawin and Wood, \cite{SawinWood}. In \S \ref{sec:randomoperators}, we define other transition operators, and show that a certain Markov chain is reversible. These results will imply that the process $coker(Haar_{n,n+u})$ is also a reversible Markov chain. In \S \ref{sec:relationwitharxiv}, we comment on the relation of the transition operators defined in \cite{arxiv} to the transition operators defined in this paper. Lastly, in \S \ref{sec:othercomments}, we discuss what happens when we change the ring $R$.

\subsection{Previous work}

The finite field case appears in many papers and it is difficult to find a precise attribution. The $\Z_p$ case was studied by Roger Van Peski, who computed, in particular, the joint distribution of 
$
coker(Haar_{n,N}) \text{ and } coker(Haar_{n,N+k})
$
in \cite[Theorem 1.3, part 2]{Van_Peski_2021}.
\newline
\newline
\noindent
Crucial to the results in this paper is the equivalence relation (\ref{eqn:eqrelation}). A very similar equivalence relation was used, independently, by Jiahe Shen \cite{shenuniversality} in his very interesting proof of quantitative universality for symmetric and anti-symmetric $p$-adic matrices.

\subsection{Acknowledgements}

The general proof of reversibility was inspired by a fruitful e-mail exchange with Roger Van Peski. The author would like to thank Alex Cowan and Jerry Wang for the invitation to speak about this research and about \cite{arxiv}, at the Waterloo number theory seminar in February 2026. The author would also like to thank Alexander Van Werde for numerous fruitful discussions, and Alexander Yu for his encouragement during the writing of this paper. 
\newline
\newline
\noindent
AI was not used in the course of this project.

\section{The equivalence classes}
\label{sec:eqrelation}

$(R,\mathfrak{m})$ is a complete local ring with a finite residue field. $G \defeq GL(R)$. 

We recall that, in \S \ref{sec:approach}, we have defined an equivalence relation, where $M_1 \sim M_2$ if and only if there exist $n$ and $m$ such that:
\[
\begin{mat}{c|c}
M_1& 0 \\ \hline
0  & \mathbf{I}_n
\end{mat}=
g^T
\begin{mat}{c|c}
M_2& 0 \\ \hline
0  & \mathbf{I}_m
\end{mat}g' \hspace{0.7in} \text{ for some } g,g' \in GL
\]

\begin{definition}
For $u \in \Z$, $\Cr$ is the set of equivalence classes of matrices such that $ \#\text{cols } - \#\text{rows } = u$
\end{definition}

Recall that, given a matrix $M$, we denote by $\Coker(M)$ the equivalence class to which $M$ belongs.

\begin{remark}
$\Cr$ has a distinguished element, i.e. the equivalence class of
\[
\begin{mat}{c}
1
\end{mat}
\text{ or }
\begin{mat}{c}
1 \\
\mathbf{0}
\end{mat}
\text{ or }
\begin{mat}{cc}
1 & \mathbf{0}
\end{mat}
\]
We will denote this equivalence class as $O_{u}(R)$.
\end{remark}

\begin{lemma}
\label{lem:basechange:k}
A matrix $M$ over $R$ belongs to the equivalence class $O_{u}(R)$ if and only if $M \mod \mathfrak{m}$ belongs to $O_{u}(k)$.
\end{lemma}

\begin{proof}
A square matrix over $R$ is invertible if and only if its reduction $\mod \mathfrak{m}$ is invertible. \autoref{lem:basechange:k} can be deduced from this fact using row and column operations.
\end{proof}

\section{The cokernel map}
\label{sec:cokernel}

\begin{theorem}
\label{thm:cokernel}
The map
\[
C_u \xrightarrow{\text{coker}} R-\mathbf{mod}
\]
is well-defined and injective
\end{theorem}

\begin{proof}
It follows from the definition of the equivalence relation that the map is well-defined. The injectivity follows from the uniqueness of minimal free resolutions of modules over local rings, \cite[Theorem 20.2]{Eisenbud}.
\end{proof}

\begin{lemma}
$\coker(O_u(R)) \cong R^{|u|}$ if $u<0$ and $\coker(O_u(R))$ is trivial if $u\geq 0$.
\end{lemma}

\begin{proof}
This follows from the definition of $O_u(R)$.
\end{proof}


\section{A random operator and a Markov chain on
$\mathcal{C}_u$}
\label{sec:markovchain}

First, we specify what we mean by a \textit{random operator} from a set $X$ to a set $Y$.

\begin{definition}
A random operator from $X$ to $Y$ is a linear operator from measures
on $X$ to measures on $Y$, that sends probability measure to
probability measures.
\end{definition}

\begin{remark}
The pro-finite ring $R$ is equipped with a canonical measure,
alternatively called the uniform measure, or the additive Haar measure.
\end{remark}

\begin{definition}
We say an $R$-valued random variable is Haar random if it is
distributed according to the Haar measure on $R$.
\end{definition}

In this section, we will write $*$ to denote a Haar random entry in a
matrix. If $*$ appears more than once, the corresponding entries will
be \textit{independent} Haar random variables.

Consider the random operator from $\text{Mat}(R)$ to $\text{Mat}(R)$,
defined as
\begin{equation}
\label{eqn:deltar}
M \rightarrow
\begin{mat}{ccc|c}
&&&*
\\
&M&&\vdots
\\
&&&*
\\ \hline
*&\hdots&*&*
\end{mat}
\end{equation}

\begin{theorem}
\label{thm:deltar}
(\ref{eqn:deltar}) defines a random operator from $\Cr$ to $\Cr$, for all $u$.
\end{theorem}

\paragraph{Proof of \autoref{thm:deltar}}

We will prove \autoref{thm:deltar} by proving \autoref{lem:deltar} and \autoref{lem:ginvariance}. Together, these will imply \autoref{thm:deltar}.
We will write
\begin{equation}
\label{eqn:tilde}
M_1 \sim M_2
\end{equation}
  to denote that $M_1$ and $M_2$ are in the same class in $\Cr$. If
$M_1$ and $M_2$ are random matrices, (\ref{eqn:tilde}) will mean that the two matrices induce the same distribution on $\Cr$. 

\begin{lemma}
\label{lem:deltar}
\[
\begin{mat}{ccc|ccc|c}
&& & && &*
\\
&M& & &\mathbf{0}& &\vdots
\\
&& & &&  &*
\\\hline
&& & &&  &*
\\
&\mathbf{0}& & &\mathbf{I}&  &\vdots
\\
&& & &&  &*
\\ \hline
*&\hdots&* & * &\hdots& * &*
\end{mat}
\sim
\begin{mat}{ccc|ccc|c}
&& & && &*
\\
&M& & &\mathbf{0}& &\vdots
\\
&& & &&  &*
\\\hline
&& & &&  &
\\
&\mathbf{0}& & &\mathbf{I}&  &\mathbf{0}
\\
&& & &&  &
\\ \hline
*&\hdots&* & &  \mathbf{0}&  &*
\end{mat} \sim 
\]
\[
\sim 
\begin{mat}{ccc|c}
&& &*
\\
&M& &\vdots
\\
&&  &*
\\ \hline
*&\hdots&* & *
\end{mat}
\]
\end{lemma}

\begin{proof}
It suffices to prove
\[
\begin{mat}{ccc|c}
&&&*
\\
&\mathbf{I}&&\vdots
\\
&&&*
\\ \hline
*&\hdots&*&*
\end{mat}
\sim
\begin{mat}{ccc|c}
&&&0
\\
&\mathbf{I}&&\vdots
\\
&&&0
\\ \hline
0&\hdots&0&*
\end{mat}
\]

Observe that
\[
\begin{mat}{ccc|c}
&&&
\\
&\mathbf{I}&&\mathbf{v}
\\
&& &
\\ \hline
&\mathbf{w}^{T}&&c
\end{mat}
\
\begin{mat}{ccc|c}
&&&
\\
&\mathbf{I}&&\mathbf{0}
\\
&& &
\\ \hline
&\mathbf{0}&&c - \mathbf{w}^{T} \mathbf{v}
\end{mat}
\]
Now suppose that $c$, $\mathbf{v}$ and $\mathbf{w}$ are independent
and Haar random. In particular $c$ is Haar random and independent of
$\mathbf{w}^T \mathbf{v}$.
\begin{claim}
The sum of two independent random variables is Haar random if one of them is Haar random.
\end{claim}
The claim is a consequence of the fact the Haar distribution is
invariant by translations. It follows that $c - \mathbf{w}^{T}
\mathbf{v}$ is Haar random. This proves \autoref{lem:deltar}.
\end{proof}

\begin{lemma}
\label{lem:ginvariance}
\[
\begin{mat}{ccc|c}
&&& *
\\
&g^T M g'&&\vdots
\\
&& &  *
\\ \hline
*&\hdots&*&*
\end{mat}
\sim
\begin{mat}{ccc|c}
&&& *
\\
&M &&\vdots
\\
&& &  *
\\ \hline
*&\hdots&*&*
\end{mat}
\]
\end{lemma}
\begin{proof}
It suffices to show:
\[
\begin{mat}{ccc|c}
&&& *
\\
&g^T M g'&&\vdots
\\
&& &  *
\\ \hline
*&\hdots&*&*
\end{mat}=
\begin{mat}{ccc|c}
&&& 0
\\
&g&&\vdots
\\
&& &  0
\\ \hline
0&\hdots&0&1
\end{mat}^T
\begin{mat}{ccc|c}
&&& *
\\
&M &&\vdots
\\
&& &  *
\\ \hline
*&\hdots&*&*
\end{mat}
\begin{mat}{ccc|c}
&&& 0
\\
&g'&&\vdots
\\
&& &  0
\\ \hline
0&\hdots&0&1
\end{mat}
\]
But this follows from the $GL_n(R)$-invariance of the Haar measure on $R^n$, for any $n$. This proves \autoref{lem:deltar} and \autoref{thm:deltar}.
\end{proof}

\begin{definition}
We denote by $\Delta_u$ the random operator induced by
(\ref{eqn:deltar}) on $\Cr$.
\[
M \rightarrow
\begin{mat}{ccc|c}
&&& *\\
&M&& \vdots\\
&&& *\\ \hline
*&\hdots&*& *
\end{mat}
\]
\end{definition}

\paragraph{$\mathbf{coker\Big(Haar_{n,n+u}\Big)}$ is a Markov chain.}

We give a corollary of \autoref{thm:deltar}

\begin{definition}
Recall that $Haar$ be an infinite matrix, whose entries are indexed by
$\mathbb{N} \times \mathbb{N}$, with Haar random variables. Recall also that
$Haar_{n,n+u}$ denotes the top-left $n  \times n+u$ minor.
\end{definition}

\begin{Corollary}
\label{cor:one}
The random process $\coker\Big(Haar_{n,n+u}\Big)$ is a Markov chain, generated
by the operator $\Delta_u$.
\end{Corollary}

\section{Positive recurrence of the Markov chain $\coker\Big(Haar_{n,n+u}\Big)$}
\label{sec:recurrence}

\begin{theorem}
\label{thm:recur}
The Markov chain $\Coker(Haar_{n,n+u})$ is positive recurrent. Equivalently, the Markov chain \[ \coker(Haar_{n,n+u}) \] is positive recurrent.
\end{theorem}

We prove this by reducing to a question over the finite field $k$.
\begin{lemma}
\label{lem:recur}
$O_u(k)$ is a positive recurrent state for the Markov chain
\[
\Coker (Haar_{n,n+u}(k))
\]
\end{lemma}

\begin{proof}
We prove \autoref{lem:recur} using two standard computations over finite fields:

\begin{itemize}
\item If $u \geq 0$
\[ 
\prob(Haar_{n,n+u}(k) \in O_u(k)) =
\]
\[
\prob(
corank(Haar_{n,n+u}( k)) = 0
)
=
\prod_{i=u+1}^{n+u} \left(1 - \frac{1}{|k|^i}\right)
\]
\item If $u<0$
\[
\prob(Haar_{n,n+u}(k) \in O_u(k)) =
\]
\[
\prob(
corank(Haar_{n,n+u}(k)) = |u|
)
=
\prod_{i=|u|+1}^{n} \left(1 - \frac{1}{|k|^i}\right)
\]
\end{itemize}

Hence,
\[
\lim_{n\rightarrow \infty} \prob(Haar_{n,n+u}(k) \in O_u(k))
=
\prod_{i=|u|+1}^\infty \left(1 - \frac{1}{|k|^i}\right)
\]
\end{proof}

\autoref{thm:recur} now follows from \autoref{lem:recur} by \autoref{lem:basechange:k}.

\begin{corollary}
The operator $\Delta_u$ has a unique stationary measure and the distribution of $c_u \Big( Haar_{n,n+u}(R) \Big)$ converges to this stationary measure.
\end{corollary}

We denote this measure as $\mu_u$.

\paragraph{Explicit form of $\mu_\uparameter(R)$}
\begin{remark}
From the work of Sawin and Wood, \cite[Lemma 6.6, Lemma 6.7]{SawinWood}, it is possible to give explicit expressions for the measures $\mu_\uparameter(R)$, as a measure on modules, when $\uparameter \geq 0$. The general expression is given in \cite[Lemma 6.7]{SawinWood}. For the statement that their measure and our measure are the same, see the last line of the penultimate paragraph of \cite[page 48]{SawinWood}.
\end{remark}

\section{More random operators; the proof of \autoref{thm:maintheorem}}
\label{sec:randomoperators}

In this section, we complete the proof of \autoref{thm:maintheorem}. \autoref{thm:maintheorem} will follow by combining \autoref{cor:two} with \autoref{cor:one} and \autoref{thm:cokernel} in the previous section.

 Consider the random operators:
 \begin{equation}
 \label{eqn:defd}
 M
 \rightarrow
 \begin{mat}{ccc}
 && \\
 &M& \\
 && \\ \hline
 *&\hdots&*
 \end{mat}
 \text{ and }
 M
 \rightarrow
 \begin{mat}{ccc|c}
 &&&*\\
 &M&&\vdots \\
 &&&* \\
 \end{mat}
 \end{equation}

 \begin{theorem}
 The two operators in (\ref{eqn:defd}) descend to random operators:
 \[
 d_{\uparameter-1,\uparameter}: \CC_\uparameter \rightarrow \CC_{\uparameter-1}
 \text{ and }
 d_{\uparameter+1,\uparameter,}: \CC_\uparameter \rightarrow \CC_{\uparameter+1},
 \]
 respectively.
 \end{theorem}

 \begin{proof}
 The proof is analogous to the proof of \autoref{thm:deltar} above.
 \end{proof}

 \begin{definition}
 We can similarly define:
 \begin{equation}
 \label{eqn:defdrs}
 d_{r,u}: \, \CC_u \rightarrow \CC_r
 \end{equation}
 for any $u$ and $r$ in $\Z$. If $u>r$, then (\ref{eqn:defdrs}) is
 defined by adding $|u-r|$ random rows. If $u<r$, (\ref{eqn:defdrs})
 is defined by adding $|u-r|$ random columns. If $u=r$,
 (\ref{eqn:defdrs}) is the identity operator.
 \end{definition}

\paragraph{Identities satisfied by $d_{r,u}$}

\begin{theorem}
\label{thm:randomoperators}
The operators $d_{r,u}$ satisfy the following identites:
\begin{itemize}
\item[(A)] $d_{r,u} d_{u,s} = d_{r,s}$ whenever $\rparameter<\uparameter<s$ or $ \rparameter >\uparameter > s$
\item[(B)] $d_{u,r} d_{r,u} = \Delta_u^{|u-r|}$; in particular $d_{u,u+1}d_{u+1,u}=d_{u,u-1}d_{u-1,u}$
\item[(C)] $d_{r,u}$ is the adjoint of $d_{u,r}$, with respect to the measures $\mu_u$ and $\mu_r$.
\end{itemize}
\end{theorem}

\begin{proof}
The first two properties follow immediately from the definitions. The last property is a consequence of \autoref{thm:rev} below.
\end{proof}

 \paragraph{Another Markov Chain}

 Given any $u\neq r$, the pair of operators
 \begin{equation}
 \label{eqn:drs}
 (d_{u,r}, d_{r,u})
 \end{equation}
 defines a Markov chain on the state space:
 \[
 \CC_u \cup \CC_r
 \]
 \begin{theorem}
 \label{thm:rev}
 The operator (\ref{eqn:drs}) defines a \textit{reversible} Markov chain on
 $\CC_u \cup \CC_r$.
 \end{theorem}

By taking the square of the operator (\ref{eqn:drs}), we get the following corollary of \autoref{thm:rev}:
 \begin{Corollary}
 \label{cor:two}
The operator $\Delta_\uparameter$ defines a \textit{reversible} Markov chain on
 $\CC_\uparameter$.
\end{Corollary}

 We now prove \autoref{thm:rev}. Without loss of generality, suppose that
 \[r>u.\]

 \begin{lemma}
 \label{lem:rev}
 Suppose that $\probM$ is a $n \times n+\uparameter$ random matrix, whose
 distribution is invariant under the action of $GL_{n} \times
 GL_{n+\uparameter}$. Then the total variation distance between the distributions of:

\[
\Coker \halfrandommatrix{\probM} \hspace{0.25in}
\text{ and }  \hspace{0.25in}
\Coker \onedoublematrix{\probM}{\mathbf{0}}{\mathbf{I}_{r-u}} 
\]

 is bounded by $o(n)$.
 \end{lemma}

\paragraph{Deduction of \autoref{thm:rev} from \autoref{lem:rev}}

First of all, observe that elementary column operations imply that:
\[
\onedoublematrix{\probM}{\mathbf{0}}{\mathbf{I}_{r-u}} 
\sim
\begin{mat}{ccc}
&& \\
&& \\
& 
\substack{
\text{ top left} 
\\
( n-(r-u)) \times (n+u) 
\\
\text{corner of $\probM$}
}
& \\
&& \\
&&
\end{mat}
\]

Now let $G_1 \in \CC_r$ and $G_2 \in \CC_u$. Take $\probM$ to be the random matrix $Haar_{n,n+u}$, conditioned on $\Coker \Big( Haar_{n,n+u} \Big) = G_2$. \autoref{lem:rev} implies the following equality:
\[
\probP \left(
\Coker(Haar_{n,n+r}) = G_1
\Big|
\Coker(Haar_{n,n+u}) = G_2
\right)=
\]

\[
=\probP \left(
\Coker(Haar_{n-(u-r),n+u}) = G_1
\Big|
\Coker(Haar_{n,n+u}) = G_2
\right)+o(n)
\]

By the Markov property, the left hand side is 
\[
\probP\left( G_2 \xrightarrow{d_{r,u}} G_1 \right)
\]
while the right hand side is 
\[
\probP
\left(
G_1 \xrightarrow{d_{u,r}} G_2
\right)
\frac{\probP(\Coker(Haar_{n+u,n+u-r})= G_1)}
{\probP(\Coker(Haar_{n+u,n}) = G_2)}
+o(n)
\]

Taking the limit $n \rightarrow \infty$ yields:
\[
\probP (G_2 \xrightarrow{d_{r,u}} G_1) \mu_u(G_1) =
\probP (G_1 \xrightarrow{d_{u,r}} G_2) \mu_r(G_2)
\]


 \subsection{The proof of \autoref{lem:rev}}

 As the statement is about the distribution induced on $\CC_r$, we can
 apply an element of $GL$ to any matrix.
 This will not change the distribution induced on $\CC_r$. In
 particular, we can apply a Haar random element of $GL_{n+u}(R)$,
 acting on the left.
 
 \begin{definition}
Let $\mathcal{G}_{n+u}(R)$ denote a multiplicatively Haar random element of $GL_{n+u}(R)$.
\end{definition}

 Thus, it suffices to bound the total variation distance between
 the following two random matrices, of dimension $(n+u) \times (r-u)$:
 \begin{equation}
\label{eqn:GHaar}
\mathcal{G}_{n+u}(R) \begin{mat}{ccc}
&&\\
&\mathbf{0}&\\
&&\\ \hline
&&\\
&\mathbf{I}_{r-u}& \\
&&\\
\end{mat} 
 \text{ and }
 \randommatrix
\end{equation}

\begin{remark}
Note that the column vectors the matrix on the left of (\ref{eqn:GHaar}) are the first $r-u$ column vectors of $\mathcal{G}$.
\end{remark}

 \begin{lemma}
 \label{lem:technical}
 The distribution of the last $r-u$ vectors of $\mathcal{G}_{n+u}(R)$ is
 the restriction of the additive Haar measure to the set of all $(r-u)$-tuples of vectors that have full $k$-rank.
 \end{lemma}

 \begin{proof}
 
 For the proof of \autoref{lem:technical}, we will use the following standard fact, that we give without proof.
 \begin{sublemma}
 \label{sublem:orbit}
 Let $N$ be any natural number. Every $(r-u)$-tuple of vectors in $R^N$ that are $k$-independent, can be completed to an $N$-tuple of $k$-independent vectors.
\end{sublemma}

\paragraph{Proof of \autoref{lem:technical} when $R$ is finite:} Suppose $R$ is finite. Then $GL_{n+u}(R)$ is finite. Applying a uniformly random element of $GL_{n+u}(R)$ to
 \begin{equation}
 \label{eqn:iomatrix}
 \begin{mat}{ccc}
&&\\
&\mathbf{0}&\\
&&\\ \hline
&&\\
&\mathbf{I}_{r-u}& \\
&&\\
\end{mat} 
\end{equation}
 produces a uniformly random element in the orbit of $GL_{n+u}(R)$ acting on (\ref{eqn:iomatrix}). By \autoref{sublem:orbit}, the orbit of $GL_{n+u}(R)$ acting on (\ref{eqn:iomatrix}) consists of all $(r-u)$-tuples of $k$-independent vectors in $R^N$. Hence, we have proven \autoref{lem:technical} for finite $R$
 
\paragraph{Proof of \autoref{lem:technical} when $R$ is infinite:} Now suppose that $R$ is not finite. Then, by the preceding argument, we know that \autoref{lem:technical} holds modulo any power of the maximal ideal of $R$. Because $R$ is profinite, it also holds that \autoref{lem:technical} is true for $R$.

 \end{proof}

Thus, the total variation distance between the two random matrices is
twice the probability that $r-u$ Haar random vectors over $k$ are \textit{not} independent.

\begin{lemma}
\label{lem:ffestimate}
The following expression gives the probability that the rank of an $(r-u) \times (n+u) $ Haar random matrix over $k$ is less than $(r-u)$.
\[
1-\prod_{i=1}^{r-u} \left( 1 - \frac{1}{|k|^{n+u-i+1}} \right) = o(n)
\]
\end{lemma}

\begin{proof}
This is a standard calculation over finite fields.
\end{proof}

Combining \autoref{lem:ffestimate} and \autoref{lem:technical} shows that the total variation distance between the random matrix distributions (\ref{eqn:GHaar}) is $o(n)$. By the discussion preceding (\ref{eqn:GHaar}), we conclude \autoref{lem:rev}.


\section{Relation with operators defined in \cite{arxiv}}
\label{sec:relationwitharxiv}

In previous work, we defined the operators:
\begin{equation}
d, \, d^{*}
\end{equation}
as certain explicit random operators on groups. Below, we give the
relation of these operators to the ones studied in this paper.

\begin{definition}
We say a matrix over $\Z_p$ is singular if there is a non-trivial
relation between the column vectors of $M$, and a non-trivial relation
between the row vectors of $M$. If a matrix is singular, then every
matrix in its equivalence class in $\CC_u$ is also singular. 
\end{definition}

\begin{theorem}
\label{thm:rel}
We have the following relations, when $R=\Z_p$:
\begin{itemize}
\item The restriction of $d_{-1,0}$ to non-singular equivalence classes
coincides with the operator $d^{*}$, defined in \cite{arxiv}.
\item The restriction of $d_{0,-1}$ to non-singular equivalence classes
coincides with the operator $d$, defined in \cite{arxiv}.
\item Hence, the operator $\Delta_0 \defeq d_{0,-1} d_{-1,0}$, defined
above, coincides with the operator $d d^{*}$, when restricted to
non-singular equivalence classes. (In \cite{arxiv}, the operator $d
d^{*}$ is also denoted by the symbol " $\Delta_0$".)
\end{itemize}
\end{theorem}

\begin{remark}
The restriction to non-singular equivalences classes is not a serious
restriction for the study of Haar random matrices over $\Z_p$, as in
this case, the singular matrices are contained in a set of measure $0$.
\end{remark}

\begin{proof}
To prove \autoref{thm:rel}, it suffices to show that 
\[
d^{*} \Big( \coker(M) \Big) =
\coker
\begin{mat}{ccc}
&& \\
&M& \\
&& \\ \hline
*&\hdots&*
\end{mat}
\]
for all non-singular $n \times n $ matrices $M$ over $\Z_p$ and 
\[
d \Big( \coker(M) \Big)
=
\coker 
\begin{mat}{ccc|c}
&&&*\\
&M&&\vdots\\
&&&*\\
\end{mat}
\]
for all non-singular $(n +1) \times n$ matrices $M$ over $\Z_p$.
But this is shown in \cite[Theorem 1.1]{arxiv}.
\end{proof}

\section{Other comments}
\label{sec:othercomments}

\begin{remark}
Any ring homomorphism $R \rightarrow R/I$ induces maps:
\begin{equation}
\label{eqn:modImap}
red_I: \, C_\uparameter (R) \rightarrow C_\uparameter (R/I)
\end{equation}
We will call this the \textit{reduction} map. The push-forward of reduction induces an operator from measures on $C_\uparameter (R)$ to measures on $C_\uparameter (R/I)$. Moreover, we can define
the adjoint of reduction, with respect to the measures $\mu_\uparameter(R)$ and
$\mu_\uparameter(R/I)$. The adjoint of reduction will be a random operator:
\[
ind_I: \, C_\uparameter(R/I) \rightarrow C_\uparameter(R)
\] 
We will call this the \textit{induction} operator. We will define a more explicit operator in the following lemma; we will then show that this explicit operator coincides with the induction operator. 
\end{remark}

\begin{lemma}
\label{lem:inductionop}
The following is a well-defined random operator from $C_u(R/I)$ to $C_u(R)$.
Given a matrix $M$ over $R/I$, replace every entry of $M$ by an independent Haar random element of its $I$-coset in $R$.
\end{lemma}

\begin{Remark}
\label{rem:adj}
On matrices equipped with the Haar measure, this operator is the adjoint of reduction modulo $I$.
\end{Remark}

\begin{proof}
By arguments similar to those in the proof of \autoref{thm:deltar}, we can
verify that this induces a random operator from $C_\uparameter(R/I)$ to
$C_\uparameter(R)$. 
\end{proof}

\begin{theorem}
\label{thm:opcommutativity}
The operator $red_I$, and the operator defined in \autoref{lem:inductionop} commute with the operators $d_{\uparameter,\rparameter}$.
\end{theorem}

\begin{proof}
The commutativity can be verified on the level of matrices.
\end{proof}

\begin{corollary}
The operator defined in \autoref{lem:inductionop} is the same as the operator $ind_I$.
\end{corollary}

\begin{proof}
Denote the operator defined in  \autoref{lem:inductionop} as $\gamma_I$. First of all, it easy to see that $O_u(R)=\gamma_I O_u(R/I)$ and $O_u(R/I)=red_I O_u(R)$. By \autoref{thm:opcommutativity}, $\Delta_u^N O_u(R)=\gamma_I \Delta_u^N O_u(R/I)$ and $\Delta_u^N O_u(R/I)=red_I \Delta_u^N O_u(R)$.\footnote{The preceding two identities are also a restatement of \autoref{rem:adj}.} Now taking the limit $N\rightarrow \infty$ gives:  $\mu_u(R)=\gamma_I \mu_u(R/I)$ and $\mu_u(R/I)=red_I \mu_u(R)$. Hence, $\gamma_I$ is the adjoint of reduction with respect to the measures $\mu_u(R)$ and $\mu_u(R/I)$ 
\end{proof}

\begin{corollary}(Of \autoref{thm:opcommutativity})
The image of $ind_{I}$ is a subspace of the space of signed measures on 
$C_{ \, \boldsymbol{\cdot} \, }(R)$
that is invariant under the action of the operators 
$d_{ \, \boldsymbol{\cdot} \, , \, \boldsymbol{\cdot} \, }$;
the kernel of $red_{I}$ is also a subspace of the space of signed measures on $C_{ \, \boldsymbol{\cdot} \, } (R)$
that is invariant under the action of the operators 
$d_{ \, \boldsymbol{\cdot} \, , \, \boldsymbol{\cdot} \, }$.
\end{corollary}

\begin{mysection}
\paragraph{Universality}
\begin{remark}
In a forthcoming paper, we will show a universality theorem for the
above operators, assuming $R$ is finite. The universality result will
allow us to prove the universality of cokernels of random matrices
with indepedent identically distributed random entries, satisfying
certain non-degeneracy conditions.
\end{remark}
\end{mysection}

\bibliographystyle{alpha}
\bibliography{ThesisBibliography}

\end{document}